\documentclass[preprint,12pt]{elsarticle}
\usepackage{geometry}
\usepackage{graphicx,psfrag, times}
\usepackage{float}
\usepackage{rotating}
\usepackage{epsfig}
\usepackage{epstopdf}
\usepackage{amsmath}
\usepackage{amssymb}
\usepackage{mathtools}
\usepackage{mathrsfs}
\usepackage{bm}
\usepackage{esint}
\usepackage{caption}
\usepackage{subcaption}
\usepackage{amsthm}
\newtheorem{definition}{Definition}[section]
\newtheorem{Proposition}{Proposition}[section]
\newtheorem{Theorem}{Theorem}[section]

\newtheorem{Lemma}{Lemma}[section]
\newtheorem{Remark}{Remark}[section]
\numberwithin{equation}{section}
\DeclareCaptionFormat{linenumber}{\begingroup\internallinenumbers#1#2#3\endgroup}
\usepackage{enumitem}
\usepackage{booktabs}
\usepackage{multirow}
\usepackage{longtable}
\usepackage{makecell}
\usepackage[table]{xcolor}
\usepackage{color}
\usepackage{pifont}
\usepackage{appendix}
\usepackage{comment}\usepackage{algorithm}
\usepackage{algpseudocode}
\usepackage{lipsum}
\setcitestyle{citesep={,},aysep={}}
\usepackage{breakcites}
\usepackage[colorlinks=true,breaklinks=true]{hyperref}
\AtBeginDocument{%
    \hypersetup{
        citecolor=blue,
        linkcolor=blue,
        urlcolor=blue,
    }%
}

\begin{document}
\begin{frontmatter}
\title{Double-Phase Neumann Problems with Variable Growth on Metric Measure Spaces}
\author[inst1]{Ritul Dutta\fnref{equal}}
\ead{rituldutta1999@gmail.com}
\affiliation[inst1]{organization={Department of Mathematics, Indian Institute of Technology Ropar},
           city={Rupnagar},
            postcode={140001},
            state={Punjab},
            country={India}}

\author[inst2]{M. Ashraf Bhat\fnref{equal}}
\ead{ashraf74267@gmail.com}
\affiliation[inst2]{organization={National Institute of Electronics and Information Technology Ropar},
           city={Rupnagar},
            postcode={140001},
            state={Punjab},
            country={India}}
            
\author[inst1]{ G. Sankara Raju Kosuru\corref{cor1}}
\ead{raju@iitrpr.ac.in}
\cortext[cor1]{Corresponding author}

\fntext[equal]{These two authors contributed equally to this work.}

\begin{abstract}
We study a variable-exponent double-phase Neumann problem on a domain in a metric measure space, where the boundary measure satisfies an upper codimension-\(\theta\) bound for \(0\le\theta<1\). Under a local comparability condition on the double-phase growth function, we establish a modular estimate for a fractional maximal operator and use it to prove the existence and boundedness of the associated trace operator. We then establish coercivity of the associated energy functional on the normalized double-phase Newtonian space and, under reflexivity, prove the existence of minimizers. We further show that the set of minimizers is closed and convex, and that any two minimizers have the same minimal weak upper gradient almost everywhere. Finally, we prove the stability of minimizers under perturbations of the Neumann data.
\end{abstract}


\end{frontmatter}

%%%%%% Main Text %%%%%%

\section{Introduction}
The study of nonlinear partial differential equations with nonstandard growth has developed significantly over the past several decades (see, for example, \cite{BaroniColomboMingione2018, ColomboMingione2015, ByunOh2020, Marcellini1989, Marcellini1991, zbMATH05711925} and references therein). Such growth conditions arise naturally in models describing heterogeneous media, where the energy density may vary spatially and cannot be adequately described by a single fixed power. A fundamental example of nonstandard growth is the double-phase functional
\[\mathcal{H}(x,\xi)=\xi^{p(x)}+a(x)\xi^{q(x)},\qquad 1<p(x)<q(x)<\infty.\]
The double-phase model was first introduced for constant exponents (\(p(x)=p\) and \(q(x)=q\)) in \cite{Zhikov}. It exhibits two distinct growth regimes depending on the behaviour of \(a(\cdot)\) and is closely related to the Lavrentiev phenomenon \cite{zbMATH01206281}. More generally, the development of variable-exponent and Musielak–Orlicz spaces has provided a flexible framework for studying variational integrals with spatially variable growth. Comprehensive accounts of these spaces and their applications can be found in \cite{DieningHarjulehtoHastoRuzicka2011, FanZhao2001, Musielak1983,KovacikRakosnik1991,HarjulehtoHasto2019}. To study such variational problems in a non-Euclidean setting, Newtonian spaces provide a natural counterpart of Sobolev spaces on metric measure spaces, with weak upper gradients replacing classical gradients \cite{HeinonenKoskela1998, Shanmugalingam2000, BjornBjorn2011,zbMATH01535638}. Variable-exponent and Musielak–Orlicz versions of Newtonian spaces have also been developed \cite{Harjulehto2007,Ohno2015}.

Neumann problems have been extensively studied in the Euclidean setting (see, for instance, \cite{KapustyanKogut2010, RyzhakovSetukha2011,zbMATH05770527,zbMATH07040479,zbMATH08182069} and references therein), as well as in more general geometric settings, including Carnot groups \cite{Nhieu2001} and the Heisenberg group \cite{DubeyKumarMishra2016}. More recently, variational approaches to Neumann problems have been developed in metric measure spaces for the \(p\)-Laplacian \cite{Maly2018,Nastasi2022}. Motivated by these developments, we consider a double-phase Neumann problem with variable growth on a metric measure space \((X,d,\mu)\). Formally, in the Euclidean setting, the corresponding problem is given by
\begin{equation*}
\begin{cases}
\displaystyle
\operatorname{div}\left(
|\nabla u|^{p(x)-2}\nabla u
+a(x)|\nabla u|^{q(x)-2}\nabla u
\right)=0,
& \text{in }\Omega,\\[1.2ex]
\displaystyle
-\left(
|\nabla u|^{p(x)-2}\nabla u
+a(x)|\nabla u|^{q(x)-2}\nabla u
\right)\cdot\eta=f,
& \text{on }\partial\Omega,
\end{cases}
\end{equation*}
where \(\Omega\subset\mathbb{R}^N\), \(N\geq2\), is a bounded Lipschitz domain and \(a\) is a non-negative bounded function. The associated energy functional is
\begin{equation*}
\mathcal{J}_f(u)
=
\int_{\Omega}
\left(
\frac{|\nabla u|^{p(x)}}{p(x)}
+a(x)\frac{|\nabla u|^{q(x)}}{q(x)}
\right)\,dx
+\int_{\partial\Omega}fu\,dS.
\end{equation*}
The metric formulation of such a Neumann problem is substantially more delicate than its Euclidean counterpart, as it requires a compatible treatment of the metric geometry, variable double-phase growth, and boundary measure. Unlike the metric \(p\)-Laplacian setting considered in \cite{Maly2018}, where the boundary term is formulated with respect to the perimeter measure, our formulation allows for a boundary measure satisfying a codimension-\(\theta\) condition with \(\theta\in[0,1)\). The boundary functional associated with the Neumann datum must be well defined and continuous on the double-phase Newtonian space \(N^{1,\mathcal H}(\Omega)\). Consequently, a trace estimate compatible with the variable double-phase modular is needed to formulate the associated variational problem. To establish such an estimate, controlling the local behavior of \(\mathcal H\) becomes essential. We impose a local comparability condition on \(\mathcal H\), which allows the variable-growth modular to interact with a fractional maximal operator. The trace estimate allows us to formulate the variational problem. As in the classical Neumann problem, the energy functional is invariant under the addition of constants. We therefore work on the normalized space
\[
N^{1,\mathcal H}_*(\Omega)
=
\left\{
u\in N^{1,\mathcal H}(\Omega):
\int_\Omega u\,d\mu=0
\right\}.
\]
Assuming the appropriate global \(\mathcal H\)-Poincar\'e inequality, we prove that the energy is coercive on the normalized space. This coercivity allows us to apply the direct method of the calculus of variations. Under the additional assumption that \(N^{1,\mathcal H}(\Omega)\) is reflexive, we obtain the existence of a minimizer. The variational problem also exhibits an interesting nonuniqueness structure. Although the minimizer need not be unique, its minimal \(\mathcal H\)-weak upper gradient is uniquely determined \(\mu\)-almost everywhere. More precisely, we prove that the set of minimizers is convex and norm-closed and that any two minimizers possess the same minimal \(\mathcal H\)-weak upper gradient, \(\mu\)-almost everywhere. Finally, we study the stability of the variational problem under perturbations of the Neumann datum, showing that the variational framework retains a meaningful continuity property with respect to the boundary data.

 The paper is organized as follows. Section~2 recalls the necessary preliminaries on metric measure spaces, double-phase energies, weak upper gradients, and Newtonian spaces. Section~3 is devoted to the maximal-function and trace estimates. In Section~4, we establish the variational results, including coercivity, existence, and the structure of minimizer set. Section~5 deals with the stability of minimizers under perturbations of the boundary data.

\section{Preliminaries and Auxiliary Results}
Throughout this paper, we assume that \((X,d,\mu)\) is a complete metric measure space, where \(\mu\) is a \(\sigma\)-finite Borel regular measure on \(X\) such that \(0<\mu(B)<\infty\) for every open ball \(B:=B(x,r)\subset X\), where \(x\in X\) and \(r>0\). We say that the measure \(\mu\) is doubling on \(X\) if there exists a constant \(C\ge 1\) such that \(\mu(2B)\le C\mu(B)\) for every ball \(B\subset X\). We do not a priori assume that \(\mu\) is doubling on \(X\). However, for a given domain \(\Omega\subset X\), we require \(\mu\) to be non-trivial in the sense that \(0<\mu(B(x,r)\cap\Omega)<\infty\) for every \(x\in\overline{\Omega}\) and every \(r>0\), and to satisfy
the following restricted doubling condition.
\begin{definition} \label{def:restricted-doubling}
Let \(\Omega\subset X\) be a domain. We say that \(\mu\) satisfies the restricted doubling condition on \(\Omega\) if there exists a constant \(C\ge 1\) such that \(\mu(B(x,2r)\cap\Omega) \le C\mu(B(x,r)\cap\Omega)\) for every \(x\in\overline{\Omega}\) and every \(r>0\).
\end{definition}
From now on, we assume that \(\Omega\) is a domain in \(X\) for which \(\mu\) is non-trivial and satisfies the restricted doubling condition in the above sense. Let \(\mathscr{H}\) be a \(\sigma\)-finite Borel regular measure on \(\partial\Omega\). For \(0\le\theta<1\), we say that \(\mathscr{H}\) satisfies an upper codimension-\(\theta\) condition if there exists a constant \(C_\theta>0\) such that
\begin{equation}\label{eq:upper-codimension}
\mathscr{H}\bigl(B(x,r)\cap\partial\Omega\bigr)
\le
C_\theta\,
\frac{\mu\bigl(B(x,r)\cap\Omega\bigr)}{r^\theta}
\end{equation}
for every \(x\in\partial\Omega\) and every \(r>0\). For more details, we refer to \cite{Maly2019}. In the remainder of the paper, \(C\) denotes a generic positive constant whose value may change from one occurrence to another.
\subsection{Double-phase Newtonian Spaces}
Let \(E \subset X\) be measurable. Consider the double-phase functional \(\mathcal{H}:E\times [0,\infty)\to [0,\infty) \), defined by, \(\mathcal{H}(x,\xi)=\xi^{p(x)}+a(x)\xi^{q(x)}\), where we assume the following conditions:
\begin{equation}\label{eq 2.2}
\left\{
\begin{aligned}
&p,q:E\to[1,\infty)\ \text{are measurable functions satisfying}\\
&1<p^{-}:=\operatorname*{ess\,inf}_{x\in E}p(x)
\le p(x)\le p^{+}:=\operatorname*{ess\,sup}_{x\in E}p(x)<\infty,\\
&1<q^{-}:=\operatorname*{ess\,inf}_{x\in E}q(x)
\le q(x)\le q^{+}:=\operatorname*{ess\,sup}_{x\in E}q(x)<\infty,\\
&0\le a(\cdot)\in L^{\infty}(E),\\
&1<p(x)<q(x)<\infty,
~~ \text{for all }x\in E.
\end{aligned}
\right.
\end{equation}
From now on, we assume that the hypothesis in \eqref{eq 2.2} hold. Consequently, \(\mathcal{H}(x,\xi)\) satisfies conditions \((\Phi1) - (\Phi4)\) of \cite{FutamuraShimomura2021, OhnoShimomura2025}. Therefore, the corresponding double-phase Musielak--Orlicz space is  \[L^{\mathcal H}(E):=\left\{u\in M(E):\varrho_{\mathcal H}(u)<\infty\right\},\]
where \[\varrho_{\mathcal H}(u)=\int_E {\mathcal H}(x,|u|)\,d\mu\] and \(L^{\mathcal H}(E)\) is endowed with the Luxemburg norm 
\[\|u\|_{L^{\mathcal H}(E)}=\inf\left\{\tau>0:\varrho_{\mathcal H}\left(\frac{u}{\tau}\right)\le 1\right\}.\] 
Under this norm, \(L^{\mathcal H}(E)\) is a reflexive Banach space (see \cite[Theorem 1.3 and Theorem 1.8]{kaminska2022uniform}). The complementary function of \(\mathcal H\) is denoted by \(\widetilde{\mathcal H}\) and is defined by
\[
\widetilde{\mathcal H}(x,\xi)
:=
\sup_{s\ge0}
\left\{
s\xi-\mathcal H(x,s)
\right\},
~~ x\in E,~~ \xi\ge0.
\]
The corresponding complementary Musielak--Orlicz space is
\[
L^{\widetilde{\mathcal H}}(E)
:=
\left\{
u\in M(E):
\varrho_{\widetilde{\mathcal H}}(u)<\infty
\right\},
\]
where
\[
\varrho_{\widetilde{\mathcal H}}(u)
=
\int_E
\widetilde{\mathcal H}(x,|u|)\,d\mu.
\]
The space \(L^{\widetilde{\mathcal H}}(E)\) is equipped with the Luxemburg norm
\[
\|u\|_{L^{\widetilde{\mathcal H}}(E)}
=
\inf\left\{
\tau>0:
\varrho_{\widetilde{\mathcal H}}
\left(\frac{u}{\tau}\right)\le1
\right\}.
\]
The spaces \(L^{\mathcal H}(E)\) and
\(L^{\widetilde{\mathcal H}}(E)\) are related through the generalized
Hölder inequality (see \cite{kaminska2022uniform}). In particular, for
\(f\in L^{\mathcal H}(E)\) and
\(g\in L^{\widetilde{\mathcal H}}(E)\),
\[
\int_E |f(x)g(x)|\,d\mu(x)
\le
2\|f\|_{L^{\mathcal H}(E)}
\|g\|_{L^{\widetilde{\mathcal H}}(E)}.
\]
The following modular estimates will be useful in the sequel (see \cite[Proposition 2.13]{Crespo2022}).

\begin{Lemma}\label{ModEst}
Let \(u\in L^{\mathcal H}(E)\). Then the following modular estimates hold.
\begin{enumerate}[label=(\roman*)]\setlength{\itemsep}{0pt}
\item\(\|u\|_{L^{\mathcal{H}}(E)}=\lambda \ne 0 ~~\text{if and only if}~~ \varrho_{\mathcal H} \!\left(\frac{u}{\lambda}\right)=1.\)
\item\(\|u\|_{L^{\mathcal H}(E)}<1 \;(\mathrm{resp.}>1,=1) ~~\text{if and only if}~~ \varrho_{\mathcal H}(u)<1 \;(\mathrm{resp.}>1,=1).\)
\item If \(\|u\|_{L^{\mathcal H}(E)}<1\), then
\(\|u\|_{L^{\mathcal H}(E)}^{q^{+}} \le \varrho_{\mathcal H}(u) \le
\|u\|_{L^{\mathcal H}(E)}^{p^{-}}\).
\item If \(\|u\|_{L^{\mathcal H}(E)}>1\), then
\(\|u\|_{L^{\mathcal{H}}(E)}^{p^{-}} \le \varrho_{\mathcal{H}}(u)
\le \|u\|_{L^{\mathcal H}(E)}^{q^{+}}\).
\item
\( \|u\|_{L^{\mathcal H}(E)}\to0 ~~\text{if and only if}~~ \varrho_{\mathcal H}(u)\to 0\).

\end{enumerate}
\end{Lemma}

The notions of upper gradients and Newtonian spaces on metric measure spaces are standard. We recall the relevant definitions in the present double-phase setting (see \cite{zbMATH06625236, zbMATH06225496,  Ohno2015, FutamuraShimomura2021} and references therein). By a rectifiable curve \(\gamma\) in \(X\), we mean a nonconstant continuous map
\(\gamma:[0,\ell_\gamma]\to X\), where \(\ell_\gamma\) is its length. We denote by \(\Gamma(E)\) the family of all rectifiable curves contained in \(E\).

\begin{definition}
Let \(\Gamma\subset\Gamma(X)\). Denote by \(\mathcal F(\Gamma)\) the collection of all Borel measurable functions \(h:X\to[0,\infty]\) such that \(\int_\gamma h ds\ge 1 \) for every \(\gamma\in\Gamma\), where \(ds\) denotes the arc-length element along \(\gamma\). The \(\mathcal H\)-modulus of \(\Gamma\) is defined by
\[ \operatorname{Mod}_{\mathcal H}(\Gamma) := \inf_{h\in\mathcal F(\Gamma)}\varrho_{\mathcal H}(h).\]
If \(\mathcal F(\Gamma)=\varnothing\), we set
\(\operatorname{Mod}_{\mathcal H}(\Gamma)=\infty\). We say that a property holds for \(\mathcal H\)-almost every \(\gamma\in\Gamma(E)\) if it holds for all \(\gamma\in\Gamma(E)\setminus\Gamma_0\) for some family \(\Gamma_0\subset\Gamma(X)\) satisfying \(\operatorname{Mod}_{\mathcal H}(\Gamma_0)=0\).
\end{definition}

\begin{definition}
Let \(u:E \to [-\infty, \infty]\) be a measurable function. A  non-negative Borel  function g on \(E\)  is said to be an \(\mathcal{H}\)-weak upper  gradient of \(u\) in \(E\) if \[|u(\gamma(0))-u(\gamma(l_\gamma))|\le\int_\gamma g\,ds\] for \(\mathcal{H}\)-almost every \(\gamma\in \Gamma(E)\). Here, we use the convention that \(|(\pm \infty) - (\pm \infty)|= \infty\).
\end{definition}
If \(u\) has an \(\mathcal{H}\)-weak upper gradient in \(L^\mathcal{H}(E)\), then there exists a minimal \(\mathcal{H}\)-weak upper gradient \(g_u\in L^\mathcal{H}(E)\) of \(u\),
that is, \(g_u\le g\) \(\mu\)-a.e. for every \(\mathcal{H}\)-weak upper gradient \(g\in L^\mathcal{H}(E)\) of \(u\) in \(E\). Moreover, \(g_u\) is unique up to sets of measure zero (\cite[Theorem 4.6]{zbMATH06225496}). 

Let \(\tilde{N}^{1,\mathcal{H}}(E)\) be the set of all functions \(u\in L^\mathcal{H}(E)\) for which there exists an \(\mathcal{H}\)-weak upper gradient \(g\in L^\mathcal{H}(E)\) of \(u\) in \(E\). For \(u\in\tilde{N}^{1,\mathcal{H}}(E)\), we define
\[
\|u\|_{\tilde{N}^{1,\mathcal{H}}(E)}
=
\|u\|_{L^\mathcal{H}(E)}
+
\inf \|g\|_{L^\mathcal{H}(E)},
\]
where the infimum is taken over all \(\mathcal{H}\)-weak upper gradients \(g\) of \(u\) in \(E\). By the existence and minimality of \(g_u\), the above functional can equivalently be written as (see
\cite[Corollary 4.9]{zbMATH06225496})
\[
\|u\|_{\tilde{N}^{1,\mathcal{H}}(E)}
=
\|u\|_{L^\mathcal{H}(E)}
+
 \|g_u\|_{L^\mathcal{H}(E)}.
\]
The functions in \(\tilde{N}^{1,\mathcal{H}}(E)\) are assumed to be defined everywhere, rather than only almost everywhere. The functional \(\|\cdot\|_{\tilde{N}^{1,\mathcal{H}}(E)}\) is only a seminorm. The double-phase Newtonian space \(N^{1,\mathcal{H}}(E)\) is defined as the space of equivalence classes of functions in \(\tilde{N}^{1,\mathcal{H}}(E)\), that is,
\[
N^{1,\mathcal{H}}(E)
=
\tilde{N}^{1,\mathcal{H}}(E)/\sim,
\]
where \(u\sim v\) if and only if \(\|u-v\|_{\tilde{N}^{1,\mathcal{H}}(E)}=0\). We write \(\|u\|_{N^{1,\mathcal{H}}(E)}\) for the norm of \(u \in N^{1,\mathcal{H}}(E)\). The space \(N^{1,\mathcal{H}}(E)\), equipped with this norm, is a Banach space (\cite[Theorem 7.1]{zbMATH06625236}). 

Let \(L^1_{\mathrm{loc}}(E)\) be the space of functions on \(E\) that are integrable on bounded subsets of \(E\). For any set \(A \subset X\) satisfying \(0<\mu(A)<\infty\), we write 
\[\fint_{A} u\,d\mu= u_{A}=\frac{1}{\mu(A)}\int_{A} u\,d\mu.\]

\begin{definition}\label{Poincare inq}
Let \(u\in L^1_{\mathrm{loc}}(E)\) and let \(g\) be an \(\mathcal{H}\)-weak upper gradient of \(u\) in \(E\). We say that the pair \((u,g)\) satisfies a weak \((1,1)\)-Poincar\'e inequality on \(E\) if there exist constants
\(C>0\) and \(\lambda >1\) such that
\[\fint_{B}|u-u_B|\,d\mu \le C\,r \fint_{\lambda B}g\,d\mu \]
for every ball $B \subset E$. If the above holds for \(g=g_u\), we simply say that \(E\) supports the weak \((1,1)\)-Poincar\'e inequality.
\end{definition}
If the pair \((u,g)\) satisfies the weak \((1,1)\)-Poincar\'e inequality on \(\Omega\), one also has
\[
\fint_{B\cap\Omega}|u-u_{B\cap\Omega}|\,d\mu
\le
C r\fint_{\lambda B\cap\Omega}g\,d\mu,
\]
for every ball \(B\) centered at a point of \(\overline{\Omega}\) (\cite[Remark~2.13]{Maly2019}). Whenever we assume that \(\Omega\) supports the weak \((1,1)\)-Poincar\'e inequality, we shall also freely use this consequence.
\begin{definition}
We say that \(E\) satisfies the global \(\mathcal H\)-Poincaré inequality if there exists a constant \(C>0\) such that
\(\|u-u_E\|_{L^{\mathcal H}(E)} \le C\|g_u\|_{L^{\mathcal H}(E)} \) for every \(u\in N^{1,\mathcal H}(E)\).
\end{definition}
For \(u\in L^1_{\mathrm{loc}}(E)\), we define the modified maximal operator \(\mathcal M^E_\lambda\) by
\[
\mathcal M^E_\lambda u(x)
=
\sup_{r>0}
\frac{1}{\mu(B(x,\lambda r)\cap E)}
\int_{B(x,r)\cap E}|u(y)|\,d\mu(y),
~~ x\in E,
\]
where the supremum is taken over all \(r>0\) for which \(\mu(B(x,\lambda r)\cap E)>0\). When \(\lambda=1\), this reduces to the usual Hardy--Littlewood maximal operator on \(E\), which we simply denote by \(\mathcal M^E\). The boundedness of the modified maximal operator on \(L^{\mathcal H}\) is well known (see, for example, \cite[Theorem 6.1]{OhnoShimomura2025} and \cite[corollary 3.4]{OhnoShimomura2018}).
\begin{Theorem}
Let \(\Omega \subset X\) be bounded. Assume that \(\Omega\) supports a weak \((1,1)\)-Poincar\'e inequality, and \(\mathcal M^{\Omega}_\lambda\) is bounded on \(L^{\mathcal H}(\Omega)\). Then \(\Omega\) satisfies the global \(\mathcal H\)-Poincar\'e inequality.
\end{Theorem}

\begin{proof}
Fix \(x\in \Omega\) and set \(r_0=2\operatorname{diam}(\Omega)\), \(B_0=B(x,r_0)\). Then \(\Omega\subset B_0\). For \(i\ge0\), let \(r_i=2^{-i}r_0,~~B_i=B(x,r_i)\cap \Omega\). For \(\mu\)-almost every \(x\in \Omega\), the Lebesgue differentiation theorem gives \(u_{B_i}\rightarrow u(x)~~\text{as }i\to\infty\). Consequently,
\[
|u(x)-u_{B_0}|
\le
\sum_{i=0}^{\infty}
|u_{B_{i+1}}-u_{B_i}|.
\]
Since \(B_{i+1}\subset B_i\), we have
\[
\begin{aligned}
|u_{B_{i+1}}-u_{B_i}|
&\le
\fint_{B_{i+1}}|u-u_{B_i}|\,d\mu\\
&\le
\frac{\mu(B_i)}{\mu(B_{i+1})}
\fint_{B_i}|u-u_{B_i}|\,d\mu.
\end{aligned}
\]
By the weak \((1,1)\)-Poincar\'e inequality with dilation constant \(\lambda'\),
\[
\fint_{B_i}|u-u_{B_i}|\,d\mu
\le
C r_i
\fint_{\lambda' B_i}g_u\,d\mu,
\]
where \(\lambda' B_i= B(x,\lambda r_i) \cap \Omega\). Therefore,
\begin{equation} \label{eq:u-uB0}
|u(x)-u_{B_0}|
\le
C r_0\mathcal M^{\Omega} g_u(x)~~\text{for \(\mu\)-a.e. \(x\in \Omega\)}.
\end{equation}
We next estimate \(|u_{B_0}-u_\Omega|\). Since \(\Omega\subset B_0\),
\[
\begin{aligned}
|u_{B_0}-u_\Omega|
&=
\left|
\frac{1}{\mu(\Omega)}
\int_\Omega (u_{B_0}-u)\,d\mu
\right|\\
&\le
\frac{1}{\mu(\Omega)}
\int_\Omega|u-u_{B_0}|\,d\mu\\
&\le
\frac{\mu(B_0)}{\mu(\Omega)}
\fint_{B_0}|u-u_{B_0}|\,d\mu.
\end{aligned}
\]
For a fixed \(x_0\in \Omega\), we have \(B(x,r_0)\subset B(x_0,2r_0)\). Thus,
\[
\frac{\mu(B_0)}{\mu(\Omega)}
\le
\frac{\mu(B(x_0,2r_0))}{\mu(\Omega)}
=:C_\Omega<\infty.
\]
Applying the weak \((1,1)\)-Poincar\'e inequality to \(B_0\), we obtain
\[
\begin{aligned}
|u_{B_0}-u_\Omega|
&\le
C_E C r_0
\fint_{\lambda' B_0}g_u\,d\mu\\
&\le
C r_0\mathcal M^{\Omega} g_u(x).
\end{aligned}
\]
Hence,
\begin{equation} \label{eq:u-uOmega-M}
|u(x)-u_E|
\le
C\,\mathcal M^{\Omega} g_u(x)  ~~\text{for \(\mu\)-a.e. \(x\in \Omega\)}.
\end{equation}
By the restricted doubling condition, for the fixed \(\lambda>1\), there exists
\(C_\lambda>0\) such that
\[
\mathcal M^{\Omega} g_u(x)
\le
C_\lambda\mathcal M^{\Omega}_\lambda g_u(x).
\]
Therefore, by \eqref{eq:u-uOmega-M},
\begin{equation}
|u(x)-u_\Omega|
\le
C\,\mathcal M^{\Omega}_\lambda g_u(x)~~\text{for \(\mu\)-a.e. \(x\in \Omega\)}.
\label{eq:pointwise-poincare}
\end{equation}
Taking the \(L^{\mathcal H}(\Omega)\)-norm in
\eqref{eq:pointwise-poincare}, we obtain
\[
\|u-u_\Omega\|_{L^{\mathcal H}(\Omega)}
\le
C
\|\mathcal M^{\Omega}_\lambda g_u\|_{L^{\mathcal H}(\Omega)}.
\]
By the boundedness of \(\mathcal M^{\Omega}_\lambda\) on \(L^{\mathcal H}(\Omega)\), we get
\[
\|u-u_\Omega\|_{L^{\mathcal H}(\Omega)}
\le
C\|g_u\|_{L^{\mathcal H}(\Omega)}.
\]
\end{proof}

\section{Existence and Boundedness of the Trace Operator}

To establish the existence and boundedness of the trace operator, we need the following local comparability condition for the double-phase modular \(\mathcal H\).
\begin{definition}\label{Local comparability property}
Let \(I\subset[0,\infty)\). The double-phase modular \(\mathcal H\) is said to satisfy the local comparability condition on \(E \times I\) if there exist constants \(C\ge1\) and \(r_0>0\) such that, for every ball \(B \subset X \) with \(\operatorname{rad}(B)<r_0\), every \(x,y\in B\cap E\), and every \(\xi\in I\),
\[
\mathcal H(x,\xi)\le C\,\mathcal H(y,C\xi).
\]
If \(I=[0,\infty)\), we simply say that \(\mathcal H\) satisfies the local comparability condition on \(E\).
\end{definition}

The local comparability condition is satisfied by double-phase modulars under some natural assumptions. If \(p(x) : \equiv p\) and \(q(x) : \equiv q\) are constant and there exists a constant \(C_a\ge1\) such that \(a(x)\le C_a a(y)\) whenever \(x,y\in B\cap E\) for every ball \(B\subset X\) with \(\operatorname{rad}(B)<r_0\), then \(\mathcal H\) satisfies the local comparability condition on \(E\). In particular, this condition is satisfied if
\(0<a_0\le a(x)\le a_1<\infty\). The following proposition records some other situations in which such conditions hold.

\begin{Proposition} \label{prop:local-comparability}
Assume that \(p\) and \(q\) are locally log-Hölder continuous, namely, there exist constants \(L_p,L_q,r_0>0\) such that
\[
|p(x)-p(y)|
\le
\frac{L_p}{\log\!\left(e+\frac{1}{d(x,y)}\right)} ~~\mbox{ and }
~~
|q(x)-q(y)|
\le
\frac{L_q}{\log\!\left(e+\frac{1}{d(x,y)}\right)}
\]
whenever \(x,y\in E\) and \(d(x,y)<r_0\). Further, assume that \(a\in C^{0,\alpha}(E), ~0<\alpha\le1\), with the Hölder constant \(L_a\). Then we have the following:
\begin{enumerate}
\item[(i)] For \(0<a_0\le a(x)\le a_1<\infty\), \(\mathcal H\) satisfies the local comparability condition on \(E\times I\), where \(I=[t_0, T],~~0 <t_0 \le T<\infty\).

\item[(ii)] Let \(U\subset E\) be open, \(X_0:=\{x\in U:a(x)>0\}\), and \(0<\theta<1\). Suppose that \(\theta\left(\frac{q(x)}{p(x)}-1\right)\le\alpha~~\text{for every }x\in X_0\). Then, for every \(\gamma>0\), there exist \(C_\gamma\ge1\) and \(r_1<\min\{1,r_0\}\) such that, whenever \(B\subset X\) is a ball with \(\operatorname{rad}(B)<\frac{r_1}{2}\), \(x,y\in B\cap U\), and \(\xi\ge 1\) satisfies \(\mathcal H(x,\xi)\le  \gamma r_0^{-\theta}\), we have \(\mathcal H(x,\xi)\le \mathcal H(y,C_\gamma\xi)\). In particular, if \[I:=\bigcap_{x\in U} \left\{\xi\ge1:\mathcal H(x,\xi)\le\gamma r_0^{-\theta}\right\},\] then \(\mathcal H\) satisfies the local comparability condition on \(U\times I\).

\end{enumerate}
\end{Proposition}
\begin{proof}
\begin{enumerate}
\item[(i)] Let \(B\subset X\) be a ball with radius \(\operatorname{rad}(B) < \frac{r_0}{2}\) and let \(x,y\in B\cap E\). Since \(p\) is locally log-Hölder continuous,
\[|p(x)-p(y)| \le \frac{L_p}{\log\!\left(e+\frac1{r_0}\right)} =:E_p \]
Hence \(\xi^{p(x)} = \xi^{p(y)}\xi^{p(x)-p(y)}.\) Since \(\xi\in[t_0,T]\), \(\xi^{p(x)-p(y)} \le \max\{T^{E_p},\,t_0^{-E_p}\} =:C_1, \) which gives \(\xi^{p(x)} \le C_1\xi^{p(y)}\). Similarly,
\[|q(x)-q(y)| \le \frac{L_q}{\log\!\left(e+\frac1{r_0}\right)} =:E_q\]
Therefore,
\(\xi^{q(x)} = \xi^{q(y)}\xi^{q(x)-q(y)} \le \max\{T^{E_q},\,t_0^{-E_q}\}\,\xi^{q(y)} =:C_2\xi^{q(y)} \). Since \(a\in C^{0,\alpha}(E)\),
\( |a(x)-a(y)| \le L_a d(x,y)^\alpha \le L_a r_0^\alpha.\)
Using the lower bound \(a(y)\ge a_0\), we obtain \[a(x) \le a(y)+L_a r_0^\alpha \le \left(1+ \frac{L_a r_0^\alpha}{a_0} \right)a(y) =:C_3a(y).\]

Consequently, \(a(x)\xi^{q(x)} \le C_{2}C_{3}\,a(y)\xi^{q(y)}\). Combining the above estimates, we get
\[
\begin{aligned}
\mathcal{H}(x,\xi) &= \xi^{p(x)} + a(x)\xi^{q(x)} \\
&\le C_{1}\xi^{p(y)} + C_{2}C_{3}a(y)\xi^{q(y)} \\
&\le C\left( \xi^{p(y)} + a(y)\xi^{q(y)}\right) \\
&= C\,\mathcal{H}(y,\xi)
\end{aligned}
\]
where \(C=\max\{C_{1},C_{2}C_{3}\}\).

\item[(ii)] Fix \(\gamma>0\). Choose \(r_1\) such that \(0<r_1< \min \{1,r_0\}\). Let \(B\subset X\) be a ball with
\(\operatorname{rad}(B)<r_1/2\), \(x,y\in B\cap U\), and set \(r:=2\operatorname{rad}(B)<r_1\). Suppose \(\xi\ge1\) satisfy \(\mathcal H(x,\xi)\le\gamma r_0^{-\theta}\). Since \(\xi^{p(x)}\le\mathcal H(x,\xi)\le\gamma r^{-\theta}\), we obtain
\[
\log\xi\le \frac{|\log\gamma|}{p^-} +\frac{\theta}{p^-}\log\frac{1}{r} \le A_\gamma\log\frac{e}{r},\]
where
\[
A_\gamma:=\frac{|\log\gamma|+\theta}{p^-}.
\]
By the local log-Hölder continuity of \(p\),
\[
|p(x)-p(y)|
\le
\frac{L_p}{\log(e+1/r)}.
\]
Hence
\[
\begin{aligned}
\xi^{p(x)}
&=\xi^{p(y)}\xi^{p(x)-p(y)}\\
&\le
\xi^{p(y)}
\exp\left(
\frac{L_p A_\gamma\log(e/r)}
{\log(e+1/r)}
\right)\\ & 
\le C_1\xi^{p(y)},
\end{aligned}
\]
where \(C_1>0\) is independent of \(x,y,r,\xi\). Similarly, \(\xi^{q(x)}\le C_2\xi^{q(y)}\) for some \(C_2>0\). Moreover, by the Hölder continuity of \(a\),\(a(x)\le a(y)+L_a r^\alpha\). Therefore,
\[
a(x)\xi^{q(x)}
\le
C_2a(y)\xi^{q(y)}
+L_a r^\alpha\xi^{q(x)}.
\]
Since \(\xi^{p(x)}\le\gamma r^{-\theta}\), we have
\[
\xi^{q(x)-p(x)}
\le
B_\gamma
r^{-\theta\left(\frac{q(x)}{p(x)}-1\right)},
\]
for some \(B_{\gamma}>0\) depending only on \(\gamma\) and the structural bounds. Now, as \(\theta\left(\frac{q(x)}{p(x)}-1\right) \le \alpha\). we have,
\[
\begin{aligned}
r^\alpha\xi^{q(x)}
&=
\xi^{p(x)}r^\alpha\xi^{q(x)-p(x)}\\
&\le
B_\gamma\xi^{p(x)}
r^{\alpha-\theta\left(\frac{q(x)}{p(x)}-1\right)}.
\end{aligned}
\]
Since \(r<r_1\le1\),
\[
r^\alpha\xi^{q(x)}
\le B_\gamma\xi^{p(x)}
\le B_\gamma C_1\xi^{p(y)}.
\]
Thus,
\[
a(x)\xi^{q(x)}
\le
C_2a(y)\xi^{q(y)}
+B_\gamma C_1\xi^{p(y)}.
\]
Combining this with \(\xi^{p(x)}\le C_1\xi^{p(y)}\), we obtain \(\mathcal H(x,\xi) \le C_\gamma' \mathcal H(y,\xi)\), where \(C_\gamma'>0\) is independent of \(x,y,r,\xi\). Finally, choose \(C_\gamma\ge1\) sufficiently large so that \(C_\gamma'C_\gamma^{-p^-}\le1\) and \(C_\gamma'C_\gamma^{-q^-}\le1\). Since \(p(y)\ge p^-\) and \(q(y)\ge q^-\), \(C_\gamma'\mathcal H(y,\xi) \le \mathcal H(y,C_\gamma\xi)\). Hence \(\mathcal H(x,\xi) \le \mathcal H(y,C_\gamma\xi)\).
\end{enumerate}
\end{proof}

 Next we recall the definition the fractional maximal operator (see, for instance, \cite{HeikkinenLehrbackNuutinenTuominen2013}).
\begin{definition} \label{WFMO}
Let \(\Omega'\) be a domain in \(X\) and \(g\in L^1_{\mathrm{loc}}(\Omega')\). The fractional maximal operator
of order \(\theta\), \(0 \le \theta<1\), is defined by
\[\mathcal{M}_{\theta}g(z)= \sup_{0<r<2\operatorname{diam}(\partial\Omega')} r^{\theta}
\fint_{B(z,r)\cap\Omega}|g(x)|\,d\mu, ~~ z\in\partial\Omega' \]
\end{definition}
The following estimate for the fractional maximal operator is a key tool in establishing the existence and boundedness of traces on \(\partial\Omega\).
\begin{Lemma} \label{lem:strong-maximal-theta}
Suppose that \(\Omega\) is bounded and \(\mathcal H\) satisfies the local comparability condition on \(\overline{\Omega}\). Then for every \(g\in L^{\mathcal H}(\Omega)\), \[\int_{\partial\Omega}\mathcal H(z,\mathcal{M}_\theta g(z))\,d\mathscr H(z) \le C \int_{\Omega} \mathcal H(x,C|g(x)|)\,d\mu.\]
\end{Lemma}
\begin{proof}
Since \(\mathcal H(z,0)=0\), we may restrict the argument to points \(z\in\partial\Omega\) with \(\mathcal M_\theta g(z)>0\). For \(\lambda>0\), define \(E_\lambda = \{z\in\partial\Omega: \mathcal{M}_\theta g(z)>\lambda\}\). Then for every \(z\in E_\lambda\), there exists \(r_{z}>0\) such that \[r_{z}^\theta \fint_{B(z,r_{z})\cap\Omega} |g|\,d\mu >\lambda. \] By the 5r-covering theorem \cite[Lemma 1.7]{BjornBjorn2011}, there exists a countable disjoint subcollection \(B_{k}=B(z_{k},r_{k}) ,k=1,2...,\) for some choice of \(z_{k} \in E_{\lambda}\) such that \(E_{\lambda}\subset \bigcup_{k}(5B_{k}\cap\partial\Omega)\). Therefore,
\[\begin{aligned}
\int_{E_\lambda}
\mathcal H(z,\lambda)\,d\mathscr H(z)
&\le \sum_k \int_{5B_k\cap\partial\Omega} \mathcal H(z,\lambda)\,d\mathscr H(z).
\end{aligned}\]
Now by Definition \ref{Local comparability property}, \(\mathcal H(z,\lambda) \le C\mathcal H(z_k,C\lambda),~~ z\in 5B_k\cap\partial\Omega \). Hence, \[\int_{E_\lambda} \mathcal H(z,\lambda)\,d\mathscr H(z) \le C
\sum_k \mathcal H(z_k,C\lambda) \mathscr H(5B_k\cap\partial\Omega).\]
Using the upper codimension-\(\theta\) bound \eqref{eq:upper-codimension} and the restricted doubling condition , we get
\[\mathscr H(5B_k\cap\partial\Omega)\le C\frac{\mu(B_k\cap\Omega)}{r_k^\theta}.\]
Consequently,
\begin{equation}\label{equation 4.1}
\int_{E_\lambda} \mathcal H(z,\lambda)\,d\mathscr H(z) \le C \sum_k
\mathcal H(z_k,C\lambda) \frac{\mu(B_k\cap\Omega)} {r_k^\theta}.
\end{equation}
Next we split \(|g| = |g|\chi_{\{|g|\le c\lambda\}} + |g|\chi_{\{|g|>c\lambda\}} \). For the first part \[r_k^\theta \fint_{B_{k}\cap\Omega} |g|\chi_{\{|g|\le c\lambda\}} \,d\mu \le cR_{0}^\theta\lambda, \]
where \(R_{0}=2\operatorname{diam}(\partial\Omega)\). Choosing \(c>0\) such that \(c<\min\left\{1,\frac{1}{2R_0^\theta}\right\}\), we obtain
\[\lambda \le C r_{k}^\theta \fint_{B_{k}\cap\Omega} |g|\chi_{\{|g|>c\lambda\}} \,d\mu. \]
Therefore,
\[\frac{\mu(B_{k}\cap\Omega)} {r_{k}^\theta} \le C \frac1\lambda \int_{B_k\cap\Omega}
|g(x)| \chi_{\{|g|>c\lambda\}}\,d\mu.\]
Substituting this estimate in equation \eqref{equation 4.1}, we get
\[
\begin{aligned}
\int_{E_\lambda}
\mathcal H(z,\lambda)\,d\mathscr H(z)
&\le C
\sum_k \mathcal H(z_k,C\lambda) \frac1\lambda \int_{B_k\cap\Omega}
|g(x)| \chi_{\{|g|>c\lambda\}}\,d\mu.
\end{aligned}\]
Again by Definition \ref{Local comparability property}, \(\mathcal H(z_k,C\lambda) \le C \mathcal H(x,C\lambda), ~~ x\in B_k\cap\Omega \). Since the balls are disjoint,

\begin{equation} \label{ModLevel}
\int_{E_\lambda} \mathcal H(z,\lambda)\,d\mathscr H(z) \le C \int_{\{|g|>c\lambda\}} \frac{\mathcal H(x,C\lambda)}{\lambda} |g(x)|\,d\mu.
\end{equation}
For \(N \in \mathbb{N}\), define \(M_N(z)=\min\{ M_\theta g(z), N\}\). Clearly \(M_N \uparrow M_\theta g(z)\) as \(N \to \infty\). Let \(A^N_{j} = \{z\in\partial\Omega: 2^j< M_N(z)\le2^{j+1}\}.\) Then the sets  \(\{A^N_j\}_{j\in\mathbb Z}\) are pairwise disjoint and \(\partial\Omega =\bigcup_{j\in\mathbb Z}A^N_j\). Thus,
\[\int_{\partial\Omega} \mathcal{H}(z,M_N(z))\,d\mathscr{H}(z) = \sum_{j\in\mathbb Z} \int_{A^N_j} \mathcal{H}(z,M_N(z))\,d\mathscr{H}(z).\]
Since \(M_N(z)\le 2^{j+1}\) for every \(z\in A^N_{j}\) and \(\mathcal{H}(z,\cdot)\) is increasing, it follows that \(\mathcal{H}(z,M_N(z)) \le \mathcal{H}(z,2^{j+1}) \le C \,\mathcal{H}(z,2^j)\). Therefore, 
\[\int_{A^N_j} \mathcal{H}(z,M_N(z))\,d\mathscr H(z) \le C \int_{A^N_j}
\mathcal{H}(z,2^j)\,d\mathscr{H}(z) \le C\int_{E_{2^j}}
\mathcal{H}(z,2^j)\,d\mathscr{H}(z).\]
Combining the last two estimates, we obtain
\[ \int_{\partial\Omega} \mathcal{H}(z,M_N(z))\,d\mathscr H(z) \le C  \sum_{j\in\mathbb Z} \int_{E_{2^j}}\mathcal{H}(z,2^j)\,d\mathscr{H}(z).\]
Using the estimate \eqref{ModLevel}, we have
\[
\begin{aligned}
\int_{\partial\Omega} \mathcal H(z, M_N(z)) \,d\mathscr H(z) &\le C \sum_{j\in\mathbb Z} \int_{\{|g|>c2^j\}} \frac{\mathcal H(x,C2^j)}{2^j} |g(x)|\,d\mu.
\end{aligned}\]
By Tonelli's theorem, we may interchange the summation and integration to obtain
\[
\begin{aligned} &\int_{\partial\Omega} \mathcal H(z, M_N(z)) \,d\mathscr{H}(z) &\le C \int_\Omega |g(x)| \sum_{2^j<c|g(x)|} \frac{\mathcal H(x,C2^j)}{2^j} \,d\mu.
\end{aligned}
\]
Since \(\mathcal H(x,s\xi) \le Cs^{p^-}\mathcal H(x,\xi),~~ 0<s\le 1 \). Putting \(\xi=C|g(x)|\) and \(s=\frac{2^j}{|g(x)|}\), we get
\[\frac{\mathcal H(x,C2^j)}{2^j} \le C \frac{\mathcal H(x,C|g(x)|)}{|g(x)|}\left(\frac{2^j}{|g(x)|}\right)^{p^--1}.\]
Hence,
\[\begin{aligned}
&|g(x)| \sum_{2^j<c|g(x)|} \frac{\mathcal H(x,C2^j)}{2^j}
&\le C \mathcal H(x,C|g(x)|) \sum_{2^j<c|g(x)|} \left(\frac{2^j}{|g(x)|}\right)^{p^--1}.
\end{aligned}\]
Since \(p^->1\), the last sum is a convergent geometric series. Integrating over \(\Omega\), we have
\[\int_{\partial\Omega}\mathcal H(z, M_N(z)) \,d\mathscr H(z) \le C \int_{\Omega} \mathcal H(x,C|g(x)|)\,d\mu.\]
Letting \(N \to \infty\) and applying the monotone convergence theorem, we obtain 
\[\int_{\partial\Omega}\mathcal H(z,\mathcal{M}_{\theta}g(z)) \,d\mathscr H(z) \le C \int_{\Omega} \mathcal H(x,C|g(x)|)\,d\mu.\]
\end{proof}
We now use the above estimate, together with the weak \((1,1)\)-Poincar\'e inequality, to construct a trace operator on \(N^{1,\mathcal H}(\Omega)\).
\begin{Theorem}\label{Bounded Trace Operator}
\label{BTO}
Assume that \(\Omega\) is bounded, \(\mathcal H\) satisfies the local comparability condition on \(\overline{\Omega}\), and \(\Omega\) supports the weak \((1,1)\)-Poincar\'e inequality. Then there exists a
bounded linear operator \(T:N^{1,\mathcal H}(\Omega)\rightarrow L^{\mathcal H}(\partial\Omega)\) such that, for every \(u\in N^{1,\mathcal H}(\Omega)\),
\[
\lim_{r\to0}
\fint_{B(z,r)\cap\Omega}|u(y)-Tu(z)|\,d\mu(y)=0
\]
for \(\mathscr H\)-almost every \(z\in\partial\Omega\).
\end{Theorem}

\begin{proof}
Let \(u\in N^{1,\mathcal H}(\Omega)\). For \(z\in\partial\Omega\) and \(r>0\), define
\[T_ru(z)=\fint_{B(z,r)\cap\Omega}u\,d\mu. \]
For \(k=0,1,2,\ldots\), set \(r_k=2^{-k}r_0\), where \(r_0>0\) is to be chosen, and write \(B_k=B(z,r_k)\cap\Omega \). Since \(B_{k+1}\subset B_k\), we have
\[
\begin{aligned}
|T_{r_k}u(z)-T_{r_{k+1}}u(z)|
&=\left|\fint_{B_{k+1}}u\,d\mu-\fint_{B_k}u\,d\mu \right|\\
&\le\fint_{B_{k+1}}|u-u_{B_k}|\,d\mu. \end{aligned}\]
Applying the restricted doubling property and the weak \((1,1)\)-Poincaré inequality, we obtain
\[|T_{r_k}u(z)-T_{r_{k+1}}u(z)| \le C r_k\fint_{\lambda B_k}g_u\,d\mu .\]
Choose \(r_0\) such that \(\lambda r_0 < 2\operatorname{diam}(\partial\Omega)\). Then by Definition \ref{WFMO}, we have
\[\fint_{\lambda B_k}|g|\,d\mu \le (\lambda r_k)^{-\theta} \mathcal{M}_\theta g_u(z).\]
Consequently, 
\begin{equation} \label{Tdiff}
|T_{r_k}u(z)-T_{r_{k+1}}u(z)| \le C r_k^{1-\theta} \mathcal{M}_\theta g_u(z).
\end{equation}
Now, for \(m>k\),
\[
\begin{aligned}
|T_{r_k}u(z)-T_{r_m}u(z)| &\le \sum_{j=k}^{m-1} |T_{r_j}u(z)-T_{r_{j+1}}u(z)|\\ &\le C \mathcal{M}_\theta g_u(z) \sum_{j=k}^{m-1} r_j^{1-\theta} \\ & \le C r_k^{1-\theta}\mathcal{M}_\theta g_u(z).
\end{aligned}
\]
From Lemma \ref{lem:strong-maximal-theta}, we have \(\mathcal{M}_\theta g_u(z)<\infty \) for \(\mathscr H\)-almost every \(z\in\partial\Omega\). Since \(r_k^{1-\theta}\to 0\), \(|T_{r_k}u(z)-T_{r_m}u(z)| \rightarrow 0\) as \(k,m\to\infty\) for almost every \(z\in\partial\Omega\). Hence,\(\{T_{r_k}u(z)\}_{k=1}^{\infty}\)is a Cauchy sequence for \(\mathscr H\)-almost every
\(z\in\partial\Omega\). Define \(Tu(z)=\lim_{k\to\infty}T_{r_k}u(z) \), we have
\[
\begin{aligned} 
\fint_{B_k}|u-Tu(z)|\,d\mu &\le \fint_{B_k}|u-u_{B_k}|\,d\mu +
|T_{r_k}u(z)-Tu(z)|.
\end{aligned}
\]
Using similar calculations as above, we get
\[ \fint_{B_k}|u-u_{B_k}|\,d\mu \le C r_k^{1-\theta}\mathcal{M}_\theta g_u(z).\]
Hence,\[\lim_{k\to\infty} \fint_{B_k} |u(x)-Tu(z)|\,d\mu =0\]
for \(\mathscr H\)-almost every \(z\in\partial\Omega\). Now, let \(0<r<r_{0}\) and choose \(k\in\mathbb{N}\) such that \(r_{k+1}<r\le r_k \). Then \(r\le r_k<2r\)
and consequently
\(B(z,r)\cap\Omega \subset B_k.\)
Hence,
\[
\begin{aligned}
\fint_{B(z,r)\cap\Omega}|u(x)-Tu(z)|\,d\mu &= \frac{1}{\mu(B(z,r)\cap\Omega)}
\int_{B(z,r)\cap\Omega}|u(x)-Tu(z)|\,d\mu\\
&\le \frac{1}{\mu(B(z,r)\cap\Omega)}\int_{B_k}|u(x)-Tu(z)|\,d\mu\\
&= \frac{\mu(B_k)} {\mu(B(z,r)\cap\Omega)} \fint_{B_k}|u(x)-Tu(z)|\,d\mu.
\end{aligned}
\]
By the restricted doubling condition of \(\mu\) on \(\Omega\), \(\mu(B_k) \le \mu(B(z,2r)\cap\Omega) \le C\,\mu(B(z,r)\cap\Omega)\). Therefore,
\[\fint_{B(z,r)\cap\Omega}|u(x)-Tu(z)|\,d\mu \le C
\fint_{B_k}|u(x)-Tu(z)|\,d\mu.\]
Finally, as \(r\to 0\), necessarily \(k\to\infty\), and hence
\[\lim_{r\to0} \fint_{B(z,r)\cap\Omega} |u(x)-Tu(z)|\,d\mu =0\]
for \(\mathscr H\)-almost every \(z\in\partial\Omega\). We now show that \(T\) is bounded. Write
\[Tu(z) = \lim_{k\to\infty} T_{r_k}u(z)= T_{r_0}u(z) + \sum_{k=0}^{\infty} \left(T_{r_{k+1}}u(z)-T_{r_k}u(z)\right).\]
Using \eqref{Tdiff}, we obtain
\[|Tu(z)| \le \mathcal{M}_{0}u(z) +C\mathcal{M}_\theta g_{u}(z) \sum_{k=0}^{\infty} r_k^{1-\theta}.\] Since the series on the right side is convergent, we have
\begin{equation}\label{eq:pointwise-trace}
|Tu(z)| \le C\bigl( \mathcal{M}_{0}u(z) + \mathcal{M}_\theta g_{u}(z) \bigr).
\end{equation}
Since \(\mathcal{H}\) is subadditive and increasing, we have
\[
\int_{\partial\Omega}
\mathcal{H}\bigl(z,|Tu(z)|\bigr)\,d\mathscr{H}(z)
\le
C\int_{\partial\Omega}
\mathcal{H}\bigl(z,\mathcal{M}_{0}u(z)\bigr)\,d\mathscr{H}(z) + C \int_{\partial\Omega}
\mathcal{H}\bigl(z,\mathcal{M}_{\theta}g_u(z)\bigr)\,d\mathscr{H}(z).\]
Applying Lemma ~\ref{lem:strong-maximal-theta}, we get
\[\int_{\partial\Omega}
\mathcal{H}\bigl(z,|Tu(z)|\bigr)\,d\mathscr{H}(z)
\le C\int_{\Omega}
\mathcal{H}\bigl(x,|u(x)|\bigr)\,d\mu
+
C\int_{\Omega}
\mathcal{H}\bigl(x,g_u(x)\bigr)\,d\mu.\]
Assume that \(\|u\|_{N^{1,\mathcal{H}}(\Omega)}\le1\). Then \(\|u\|_{L^{\mathcal{H}}(\Omega)}\le1\) and \(\|g_u\|_{L^{\mathcal{H}}(\Omega)}\le1\).
By Lemma \ref{ModEst}, \(\varrho_{\mathcal{H}}(u)\le 1\) and \(\varrho_{\mathcal{H}}(g_u)\le 1\). Hence \[\int_{\partial\Omega} \mathcal{H}\bigl(z,|Tu(z)|\bigr)\,d\mathscr{H}(z)\le2C.\] Choose \(\Lambda >0\) such that
\[ \int_{\partial\Omega} \mathcal{H}\left(z,\frac{|Tu(z)|}{\Lambda}\right)\,d\mathscr{H}(z)\le1.\] By Lemma \ref{ModEst}, this yields \[\left\|  \frac{Tu}{\Lambda} \right\|_{L^{\mathcal{H}}(\partial\Omega)} \le1. \]
This completes the proof.
\end{proof}

\section{Variational Formulation and Minimization}

In this section, we introduce the variational formulation and study the minimization of the associated energy functional. Define
\[
\mathcal{A}(x,\xi)
=
\frac{\xi^{p(x)}}{p(x)}
+a(x)\frac{\xi^{q(x)}}{q(x)},
~~ \xi\ge0.
\]
For each \(x\in\Omega\), the function \(\xi\mapsto\mathcal{A}(x,\xi)\) is strictly convex and strictly increasing on \([0,\infty)\). Let \(f\in L^{\widetilde{\mathcal H}}(\partial\Omega)\) satisfy the  compatibility condition \[\int_{\partial\Omega}f\,d\mathscr H=0.\] 
We associate with \(f\) the energy functional
%\(\mathcal{J}_f:N^{1,\mathcal H}(\Omega)\to\mathbb{R}\) is given by
\[
\mathcal{J}_f(u)
=
\int_{\Omega}
\mathcal{A}(x,g_u)\,d\mu
+
\int_{\partial\Omega}f\,Tu\,d\mathscr H,~~u \in N^{1,\mathcal H}(\Omega).
\]

\begin{Proposition} The energy functional \(\mathcal{J}_f\) is convex on \(N^{1,\mathcal H}(\Omega)\).
\end{Proposition}
\begin{proof}
Let \(u,v\in N^{1,\mathcal H}(\Omega)\) and let \(\lambda\in[0,1]\).
Set \(w:=\lambda u+(1-\lambda)v.\) Since \(g_u\) and \(g_v\) are \(\mathcal H\)-weak upper gradients of \(u\) and \(v\), respectively, the function
\(\lambda g_u+(1-\lambda)g_v\) is an \(\mathcal H\)-weak upper gradient of \(w\). Consequently, we have \(g_w\le \lambda g_u+(1-\lambda)g_v ~~\mu\text{-a.e. in }\Omega\). Since \(p(x)>1\), the function
\(t\mapsto\frac{t^{p(x)}}{p(x)}\) is convex on \([0,\infty)\). Therefore, 
\begin{equation}\label{eq 5.6}
\frac{g_w^{p(x)}}{p(x)} \le \frac{\big(\lambda g_u+(1-\lambda)g_v\big)^{p(x)}}{p(x)} 
\le \lambda\frac{g_u^{p(x)}}{p(x)} + (1-\lambda)\frac{g_v^{p(x)}}{p(x)}.
\end{equation}
Similarly, since \(q(x)>1\),
\begin{equation}\label{eq 5.7}
\frac{g_w^{q(x)}}{q(x)} \le \lambda\frac{g_u^{q(x)}}{q(x)} +
(1-\lambda)\frac{g_v^{q(x)}}{q(x)}.
\end{equation}
Multiplication by \(a(x)\) gives
\begin{equation}\label{eq 5.8}
a(x)\frac{g_w^{q(x)}}{q(x)}
\le
\lambda a(x)\frac{g_u^{q(x)}}{q(x)}
+
(1-\lambda)a(x)\frac{g_v^{q(x)}}{q(x)}.    
\end{equation}
Now, adding \eqref{eq 5.6}and \eqref{eq 5.8} and integrating over \(\Omega\), we get 
\begin{equation}\label{eq5.9}
\int_\Omega \mathcal{A}(x,g_w) d\mu \le \lambda \int_\Omega \mathcal{A}(x,g_u)d\mu +(1-\lambda) \int_\Omega \mathcal{A}(x,g_v)d\mu. \end{equation}
Moreover, by the linearity of the trace operator,
\(Tw=\lambda Tu+(1-\lambda)Tv.\)
Hence \begin{equation}\label{eq5.10}
\int_{\partial\Omega}fTw\,d\mathscr H = \lambda\int_{\partial\Omega} fTu\,d\mathscr{H} ~~+ (1-\lambda)\int_{\partial\Omega}fTv\,d\mathscr H.
\end{equation}
Combining \eqref{eq5.9} and \eqref{eq5.10}, we obtain
\[
\begin{aligned}
\mathcal{J}_f(\lambda u+(1-\lambda)v) &\le \lambda\mathcal{J}_f(u)
+(1-\lambda) \mathcal{J}_f(v).
\end{aligned}
\]for every \(u,v\in N^{1,\mathcal H}(\Omega)\) and every
\(\lambda\in[0,1]\). Thus \(\mathcal{J}_f\) is convex on
\(N^{1,\mathcal H}(\Omega)\).
\end{proof}

For the remainder of this section, we work with the following normalized space
\[
N^{1,\mathcal H}_*(\Omega)
:=
\left\{
u\in N^{1,\mathcal H}(\Omega):
\int_\Omega u\,d\mu=0
\right\}.
\]
It is easy to prove that \(N^{1,\mathcal H}_{*}(\Omega)\) is a closed linear subspace of \(N^{1,\mathcal H}(\Omega)\). Consequently, if \(N^{1,\mathcal H}(\Omega)\) is reflexive, then \(N^{1,\mathcal H}_{*}(\Omega)\) is also reflexive. This reflexivity will be used in the existence theorem. A function \(u_0\in N^{1,\mathcal H}_{*}(\Omega)\) is called a variational solution associated with \(\mathcal{J}_f\) if
\[
\mathcal{J}_f(u_0)\leq \mathcal{J}_f(v) ~~\text{for every }v\in N^{1,\mathcal H}_{*}(\Omega).
\]
From now on, we write \(I(f):=\inf_{u\in N^{1,\mathcal H}_*(\Omega)}\mathcal{J}_f(u)\). Below we establish the coercivity of the energy functional on \(N^{1,\mathcal H}_{*}(\Omega)\).
\begin{Proposition}\label{coercivity}
Assume that the conditions of Theorem \ref{BTO} are satisfied and that
\(\Omega\) satisfies the \(\mathcal H\)-Poincaré inequality. Then
\(\mathcal{J}_f\) is coercive on \(N^{1,\mathcal H}_{*}(\Omega)\).
\end{Proposition}

\begin{proof}
Let \(\{u_n\}\subset N^{1,\mathcal H}_{*}(\Omega)\) be a sequence such that \(\|u_n\|_{N^{1,\mathcal H}(\Omega)}\rightarrow\infty\). Since \(\Omega\) satisfies the \(\mathcal H\)-Poincaré inequality, we have
\(\|u_n\|_{N^{1,\mathcal H}(\Omega)} = \|u_n\|_{L^{\mathcal H}(\Omega)} + \|g_{u_n}\|_{L^{\mathcal H}(\Omega)} \le  \|g_{u_n}\|_{L^{\mathcal H}(\Omega)}\). Therefore, \(\|g_{u_n}\|_{L^{\mathcal H}(\Omega)} \rightarrow\infty\). Now, \(\mathcal{A}(x,\xi) \ge \frac{1}{q^+}\mathcal H(x,\xi)\) for every \(\xi\ge0\). Consequently,
\[
\int_\Omega\mathcal A(x,g_{u_n})\,d\mu
\ge
\frac{1}{q^+}
\int_\Omega\mathcal H(x,g_{u_n})\,d\mu.
\]
Since \(\|g_{u_n}\|_{L^{\mathcal H}(\Omega)} > 1\) for all sufficiently
large \(n\), Lemma \ref{ModEst} gives
\[
\int_\Omega\mathcal H(x,g_{u_n})\,d\mu
\ge
\|g_{u_n}\|_{L^{\mathcal H}(\Omega)}^{p^-}.
\]
Thus,
\[
\int_\Omega\mathcal A(x,g_{u_n})\,d\mu
\ge
\frac{1}{q^+}
\|g_{u_n}\|_{L^{\mathcal H}(\Omega)}^{p^-}.
\]
On the other hand, we have
\[
\begin{aligned}
\left|
\int_{\partial\Omega}fTu_n\,d\mathscr H
\right|
&\le
C\|f\|_{L^{\widetilde{\mathcal H}}(\partial\Omega)}
\|Tu_n\|_{L^{\mathcal H}(\partial\Omega)}\\
&\le
C\|u_n\|_{N^{1,\mathcal H}(\Omega)}\\
&\le
C\|g_{u_n}\|_{L^{\mathcal H}(\Omega)},
\end{aligned}
\]
where the first inequality is due to the generalized Hölder inequality and second is due to Theorem \ref{BTO}. Therefore, for all sufficiently large \(n\),
\[
\mathcal{J}_f(u_n)
\ge
\frac{1}{q^+}
\|g_{u_n}\|_{L^{\mathcal H}(\Omega)}^{p^-}
-
C\|g_{u_n}\|_{L^{\mathcal H}(\Omega)}.
\]
Since \(p^->1\), Young's inequality implies that, for every
\(\varepsilon>0\), there exists \(C_\varepsilon>0\) such that
\[
C\|g_{u_n}\|_{L^{\mathcal H}(\Omega)}
\le
\varepsilon
\|g_{u_n}\|_{L^{\mathcal H}(\Omega)}^{p^-}
+C_\varepsilon.
\]
Choosing \(\varepsilon=1/(2q^+)\), we obtain
\[
\mathcal{J}_f(u_n)
\ge
\frac{1}{2q^+}
\|g_{u_n}\|_{L^{\mathcal H}(\Omega)}^{p^-}
-C',
\]
where \(C'>0\) is independent of \(n\). Since \(\|g_{u_n}\|_{L^{\mathcal H}(\Omega)}\rightarrow\infty\), it follows that \(\mathcal{J}_f(u_n)\rightarrow\infty\). Hence \(\mathcal{J}_f\) is coercive on \(N^{1,\mathcal H}_{*}(\Omega)\).
\end{proof}

We are ready to establish the existence of a minimizer for the energy functional \(\mathcal{J}_f\) over the normalized space \(N^{1,\mathcal H}_*(\Omega)\).

\begin{Theorem} \label{thm:Existence} Assume that the conditions of Proposition \ref{coercivity} are satisfied and \(N^{1,\mathcal H}(\Omega)\) is reflexive. Then there exists \(u_0\in N^{1,\mathcal H}_*(\Omega)\) such that
\(\mathcal{J}_f(u_0) = \inf_{u\in N^{1,\mathcal H}_*(\Omega)}\mathcal{J}_f(u).\)
\end{Theorem}

\begin{proof}
By Proposition \ref{coercivity}, \(\mathcal{J}_f\) is bounded from below on \(N^{1,\mathcal H}_*(\Omega)\). Hence \(-\infty<I(f)\le \mathcal{J}_f(0)=0\). Choose a sequence \(\{u_k\}_{k=1}^{\infty}\subset N^{1,\mathcal H}_*(\Omega)\) such that \(\mathcal{J}_f(u_k)\rightarrow I(f)\) as \(k\to\infty\). The coercivity of \(\mathcal{J}_f\) implies that \(\{u_k\}\) is bounded in \(N^{1,\mathcal H}(\Omega)\). In particular, \(\{u_k\}\) is bounded in \(N^{1,\mathcal H}_*(\Omega)\). Since \(N^{1,\mathcal H}_*(\Omega)\) is reflexive, there exist, after passing to a subsequence, \(u_{0}\in N^{1,\mathcal H}_*(\Omega)\) such that
\begin{equation}\label{eq5.11}
u_k\rightharpoonup u_{0}
~~\text{in }N^{1,\mathcal H}_*(\Omega).
\end{equation}
Since the embedding \(N^{1,\mathcal H}(\Omega)\hookrightarrow L^{\mathcal H}(\Omega)\) is continuous, \eqref{eq5.11} also implies \(u_k\rightharpoonup u_{0} ~~\text{in }L^{\mathcal H}(\Omega)\). For each \(k\in \mathbb{N}\), let \(g_k\) be the minimal \(\mathcal H\)-weak upper gradient of \(u_k\). Since \(L^{\mathcal H}(\Omega)\) is reflexive and \(\{g_k\}_{k=1}^{\infty}\) is bounded in \(L^{\mathcal H}(\Omega)\), after passing to a further subsequence, there exists \(g\in L^{\mathcal H}(\Omega)\) such that \(u_k\rightharpoonup u_{0} ~~\text{in }L^{\mathcal H}(\Omega)\)
and \(g_k\rightharpoonup g ~~\text{in }L^{\mathcal H}(\Omega)\). Thus, \((u_k,g_k)\rightharpoonup (u_{0},g) ~~\text{in } L^{\mathcal H}(\Omega)\times L^{\mathcal H}(\Omega)\). By Mazur's lemma, there exist integers \(N(k)\ge k\) and coefficients \(\alpha_{k,i}\ge0\), \(i=k,\ldots,N(k)\), satisfying \( \displaystyle \sum_{i=k}^{N(k)}\alpha_{k,i}=1\), such that 
\[
\widetilde u_k
=
\sum_{i=k}^{N(k)}\alpha_{k,i}u_i,
~~
\widetilde g_k
=
\sum_{i=k}^{N(k)}\alpha_{k,i}g_i
\]
satisfy \(\widetilde u_k\rightarrow u ~~\text{in }L^{\mathcal H}(\Omega)\) and \(\widetilde g_k\rightarrow g
~~\text{in }L^{\mathcal H}(\Omega)\). From the convexity of the class of \(\mathcal H\)-weak upper gradients, it follows that \(\widetilde g_k\) is an \(\mathcal H\)-weak upper gradient of \(\widetilde u_k\). Therefore, by \cite[Lemma 4.6]{MaedaOhnoShimomura2019}), \(g\) is an \(\mathcal H\)-weak upper gradient of \(u_{0}\). Hence, \(\mathcal{A}(x,g_{u_{0}})\le\mathcal{A}(x,g) ~~\text{\(\mu\)-a.e. in }\Omega.\)
Consequently,
\begin{equation} \label{5.17}
\int_\Omega \mathcal{A}(x,g_{u_{0}})\,d\mu \le \int_\Omega\mathcal{A}(x,g)\,d\mu.
\end{equation}
After passing to a subsequence if necessary, we may assume that
\(\widetilde g_k(x)\rightarrow g(x) ~~\text{for \(\mu\)-a.e. }x\in\Omega.
\) Since \(\mathcal{A} \ge0\), Fatou's lemma yields
\[\int_\Omega\mathcal{A}(x,g)\,d\mu \le \liminf_{k\to\infty}
\int_\Omega \mathcal{A}(x,\widetilde g_k)\,d\mu.\]
Therefore, by \eqref{5.17},
\begin{equation} \label{5.18}
\int_\Omega\mathcal{A}(x,g_{u_{0}})\,d\mu \le \liminf_{k\to\infty} \int_\Omega\mathcal{A}(x,\widetilde g_{k})\,d\mu.
\end{equation}
Next, since \(\mathcal{A}(x,\cdot)\) is convex,
\begin{equation} \label{5.19}
\int_\Omega\mathcal{A}(x,\widetilde g_k)\,d\mu = \int_\Omega \mathcal{A} \left(x,\sum_{i=k}^{N(k)} \alpha_{k,i}g_i \right)d\mu 
\le \sum_{i=k}^{N(k)} \alpha_{k,i} \int_\Omega\mathcal{A}(x,g_i)\,d\mu.
\end{equation}
Now, let \(B(v):=\int_{\partial\Omega}fTv\,d\mathscr{H}\). Since \(T:N^{1,\mathcal H}(\Omega)\rightarrow L^{\mathcal H}(\partial\Omega) \) is bounded and linear and \(f\in L^{\widetilde{\mathcal H}}(\partial\Omega),\)
the generalized Hölder inequality imply that \(B\) is a bounded linear functional on
\(N^{1,\mathcal H}(\Omega)\). Since \(\widetilde u_k\) is a convex combination of a tail of the sequence \(\{u_k\}_{k=1}^{\infty}\), by \eqref{eq5.11}, we also have \(\widetilde u_k\rightharpoonup u_{0} ~~\text{in }N^{1,\mathcal H}_*(\Omega).\) Consequently, \(B(\widetilde u_k)\rightarrow B(u_{0})\). Now, by \eqref{5.19},
\begin{align} \label{5.20}
\int_\Omega \mathcal{A}(x,\widetilde g_k)\,d\mu
+B(\widetilde u_k)
&\le \sum_{i=k}^{N(k)}
\alpha_{k,i}
\int_\Omega\mathcal{A}(x,g_i)\,d\mu
+
\sum_{i=k}^{N(k)}\alpha_{k,i}B(u_i) \notag\\
&=
\sum_{i=k}^{N(k)}
\alpha_{k,i}\mathcal{J}_f(u_i).
\end{align}
Therefore, by \eqref{5.18} and \eqref{5.20},
\begin{align}
\mathcal{J}_f(u_{0}) &= \int_\Omega\mathcal{A}(x,g_u)\,d\mu+B(u) \notag\\
&\le \liminf_{k\to\infty} \left[ \int_\Omega\mathcal{A}(x,\widetilde g_k)\,d\mu +B(\widetilde u_k) \right] \notag\\
&\le \liminf_{k\to\infty} \sum_{i=k}^{N(k)} \alpha_{k,i}\mathcal{J}_f(u_i) = I(f).\notag
\end{align}
Since \(u_0\in N^{1,\mathcal H}_*(\Omega)\), we have \(\mathcal{J}_f(u_{0})=I(f).\) Hence, \(u_0\) is a minimizer of \(\mathcal{J}_f\) on
\(N^{1,\mathcal H}_*(\Omega)\).
\end{proof}
\begin{Remark} \label{Remark:WLSC}
The argument above also shows that \(\mathcal J_f\) is weakly lower
semicontinuous on \(N^{1,\mathcal H}_*(\Omega)\). Indeed, let \(u_k\rightharpoonup u_0
~~\text{in }N^{1,\mathcal H}_*(\Omega)\), and set \(L:=\liminf_{k\to\infty}\mathcal J_f(u_k)\). If \(L=+\infty\), there is nothing to prove. Otherwise, by passing to a subsequence, still denoted by \((u_k)\), we may assume that
\(\mathcal J_f(u_k)\rightarrow L\). From the proof of Theorem \ref{thm:Existence}, we have
\[
\mathcal J_f(u_0)
\le
\liminf_{k\to\infty}
\sum_{i=k}^{N(k)}
\alpha_{k,i}\mathcal J_f(u_i).
\]
Since \(\mathcal J_f(u_i)\to L\), every convex combination of a tail
of the sequence also converges to \(L\). Therefore, \(\mathcal J_f(u_0) \le L =\displaystyle \liminf_{k\to\infty}\mathcal J_f(u_k)\). 
\end{Remark}
The following lemma shows that the set of minimizers of \(\mathcal{J}_f\) is convex and norm-closed.
\begin{Theorem}
\label{CCM}
Let \(\mathcal{M} := \left\{ u\in N^{1,\mathcal H}_*(\Omega):\mathcal{J}_f(u)=I(f)
\right\}\). Then \(\mathcal{M}\) is convex and norm-closed.
\end{Theorem}

\begin{proof}
Let \(u,v\in\mathcal{M}\), \(\lambda\in(0,1)\) and \(w:=\lambda u+(1-\lambda)v\). Then,
\[\int_\Omega w\,d\mu = \lambda\int_\Omega u\,d\mu + (1-\lambda) \int_\Omega v\,d\mu=0.\]
Consequently, \(w\in N^{1,\mathcal H}_*(\Omega)\). Since \(\lambda g_u+(1-\lambda)g_v\) is an \(\mathcal H\)-weak upper gradient of \(w\), \(g_w\le \lambda g_u+(1-\lambda)g_v ~~\mu\text{-a.e. in }\Omega\). We know that \(\xi\mapsto \mathcal{A}(x,\xi)\) is convex and increasing on \([0,\infty)\). Therefore,
\[
\int_\Omega\mathcal{A}(x,g_w)\,d\mu
\le
\lambda\int_\Omega\mathcal{A}(x,g_u)\,d\mu
+
(1-\lambda)\int_\Omega\mathcal{A}(x,g_v)\,d\mu.
\]
Since the trace operator \(T\) is linear, we have
\[\int_{\partial\Omega}fTw\,d\mathscr H = \lambda\int_{\partial\Omega} fTu\,d\mathscr H + (1-\lambda) \int_{\partial\Omega}fTv\,d\mathscr H.\] Adding these two estimates, we get
\(\mathcal{J}_f(w) \le \lambda \mathcal{J}_f(u)+(1-\lambda)\mathcal{J}_f(v)=I(f)\). But \(I(f)\le \mathcal{J}_f(w)\). Therefore, \(\mathcal{J}_f(w)=I(f)\) and hence \(w\in\mathcal{M}\). Thus \(\mathcal{M}\) is convex.

Now let \(\{u_n\}\subset\mathcal{M}\) and suppose that \(u_n\to u ~~\text{in }N^{1,\mathcal H}_*(\Omega)\). Then \(u\in N^{1,\mathcal H}_*(\Omega)\). Since \(\mathcal{J}_f\) is sequentially lower semicontinuous, \( \displaystyle I(f) \le \mathcal{J}_f(u) \le \liminf_{n\to\infty} \mathcal{J}_f(u_n) =I(f)\). Therefore, \(\mathcal{J}_f(u)=I(f)\) and hence \(u\in\mathcal{M}\). Thus \(\mathcal{M}\) is norm-closed.
\end{proof}

Although the minimizer of \(\mathcal{J}_f\) is not necessarily unique, the following theorem shows that any two minimizers have minimal \(\mathcal H\)-weak upper gradients that agree \(\mu\)-almost everywhere.

\begin{Theorem}\label{Theorem 4.2}
Let \(\mathcal{M} := \left\{ u\in N^{1,\mathcal H}_*(\Omega):\mathcal{J}_f(u)= I(f) \right\}\). If \(u,v\in\mathcal{M}\), then \( g_u=g_v ~~ \mu\text{-a.e. in }\Omega\).
\end{Theorem}
\begin{proof}
Let \(u,~v \in \mathcal{M}\) and \(w:=\frac{u+v}{2}\). By Lemma \ref{CCM}, \(w\in\mathcal{M}\). Moreover, by sublinearity of  \(\mathcal H\)-weak upper gradients, \(g_w\le\frac{g_u+g_v}{2} ~~\mu\text{-a.e. in }\Omega\). Set \(E:=\{x\in\Omega:g_u(x)\ne g_v(x)\}\) and assume to the contrary that \(\mu(E)>0\). Since \(\xi\mapsto \mathcal{A}(x,\xi)\) is strictly convex and strictly increasing on \([0,\infty)\),  \[\mathcal{A}(x,g_w) < \mathcal{A}\left(x,\frac{g_u+g_v}{2}\right) < \frac{\mathcal{A}(x,g_u)+\mathcal{A}(x,g_v)}{2}~~\text{for \(\mu\)-a.e.}~~ x\in E.\] Therefore,
\[\int_\Omega\mathcal{A}(x,g_w)\,d\mu < \frac{1}{2} \int_\Omega\mathcal{A}(x,g_u)\,d\mu
+\frac{1}{2} \int_\Omega\mathcal{A}(x,g_v)\,d\mu.\] By the linearity of the trace operator,
\[\int_{\partial\Omega}fTw\,d\mathscr H = \frac{1}{2} \int_{\partial\Omega} fTu\,d\mathscr H +\frac{1}{2}\int_{\partial\Omega} fTv\,d\mathscr H.\] Adding the two estimates, we obtain
\(\mathcal{J}_f(w) <\frac{1}{2}\mathcal{J}_f(u)+\frac{1}{2}\mathcal{J}_f(v) = I(f).\)  This contradicts the definition of \(I(f)\). Therefore, \(\mu(E)=0\) and hence \({g_u=g_v~~\mu\text{-a.e. in }\Omega.}\)
\end{proof}

\begin{Remark}\label{remark4.2}
If the space \(\Omega\) admits a Cheeger-type differential structure with the corresponding Hilbertian property, then the above uniqueness of the minimal \(\mathcal H\)-weak upper gradient, together with the convexity of \(\mathcal{M}\), implies that the minimizer is unique \(\mu\)-a.e. in \(\Omega\) (see \cite{Maly2018}).
\end{Remark}
\section{Stability of Minimizers with Respect to Neumann Data}
In this section, we study the dependence of minimizers on the Neumann data. This complements the results established in the previous section by examining how minimizers behave when the boundary data are varied.

\begin{Theorem} \label{thm:stability}
Assume that the conditions of Proposition \ref{coercivity} are satisfied and \(N^{1,\mathcal H}(\Omega)\) is reflexive. Let \((f_k)\subset L^{\widetilde{\mathcal{H}}}(\partial\Omega)\) satisfy \( \int_{\partial\Omega}f_k\,d\mathscr{H}=0\) for every \(k\in\mathbb N,~~f_k\to f~~\text{in }L^{\widetilde{\mathcal{H}}}(\partial\Omega)\). Suppose \(u_k\in N^{1,\mathcal{H}}_*(\Omega)\) satisfy \(\mathcal{J}_{f_k}(u_k)=I(f_k)\). Then \(\int_{\partial\Omega}f\,d\mathscr{H}=0\), and there exist \(N(k)\ge k\) and \(\lambda_{j,k}\ge0\), \(j=k,\ldots,N(k)\), with \( \sum_{j=k}^{N(k)}\lambda_{j,k}=1\), such that \(  \sum_{j=k}^{N(k)}\lambda_{j,k}u_j\) converges strongly in \(N^{1,\mathcal{H}}(\Omega)\) to some \(u\in N^{1,\mathcal{H}}_*(\Omega)\), and \(\mathcal{J}_f(u)=I(f)\).
\end{Theorem}
\begin{proof}
By the generalized H\"older inequality,
\[
\left|
\int_{\partial\Omega}(f_k-f)\,d\mathscr H
\right|
\le
C\|f_k-f\|_{L^{\widetilde{\mathcal H}}(\partial\Omega)}
\|1\|_{L^{\mathcal H}(\partial\Omega)}.
\]
Since \(\Omega\) is bounded, there exist \(x_0\in\partial\Omega\) and \(R>0\) such that \(\partial\Omega\subset B(x_0,R)\). Hence, by the upper codimension-\(\theta\) condition \eqref{eq:upper-codimension}, \(\mathscr H(\partial\Omega) \le C R^{-\theta}\mu(B(x_0,R)\cap\Omega) <\infty\). Thus \(1\in L^{\mathcal H}(\partial\Omega)\) and
\(\|1\|_{L^{\mathcal H}(\partial\Omega)}<\infty\). As \(f_k\to f\) in \(L^{\widetilde{\mathcal H}}(\partial\Omega)\), it follows that \(\int_{\partial\Omega}f\,d\mathscr H=0\). Since \(u_k\) is a minimizer of \(\mathcal J_{f_k}\) and
\(0\in N^{1,\mathcal H}_*(\Omega)\), \(I(f_k)=\mathcal J_{f_k}(u_k) \le \mathcal J_{f_k}(0)=0\). From Proposition \ref{coercivity} and the global \(\mathcal H\)-Poincar\'e inequality, we have
\[
\|u_k\|_{L^{\mathcal H}(\Omega)}
\le C\|g_{u_k}\|_{L^{\mathcal H}(\Omega)}
\le C.
\]
Consequently, \((u_k)\) is bounded in \(N^{1,\mathcal H}(\Omega)\). By reflexivity, passing to a subsequence and relabeling it as \((u_k)\), we may assume that \(u_k\rightharpoonup u~~\text{weakly in }N^{1,\mathcal H}(\Omega)\)
for some \(u\in N^{1,\mathcal H}(\Omega)\). Since \(N^{1,\mathcal H}_*(\Omega)\) is a closed subspace of
\(N^{1,\mathcal H}(\Omega)\) and \(u_k\in N^{1,\mathcal H}_*(\Omega)\), we have \(u\in N^{1,\mathcal H}_*(\Omega)\).

We next show that \(I(f_k)\rightarrow I(f)\). Fix \(\varepsilon>0\). By the definition of \(I(f)\), there exists
\(v_\varepsilon\in N^{1,\mathcal H}_*(\Omega)\) such that \(\mathcal J_f(v_\varepsilon)<I(f)+\varepsilon\). Since \(v_\varepsilon\) is admissible for \(\mathcal J_{f_k}\), \(I(f_k)\le\mathcal J_{f_k}(v_\varepsilon)\). Moreover,
\[
\mathcal J_{f_k}(v_\varepsilon)
=
\mathcal J_f(v_\varepsilon)
+
\int_{\partial\Omega}(f_k-f)Tv_\varepsilon\,d\mathscr H.
\]
By the generalized H\"older inequality and the boundedness of the trace operator,
\[
\begin{aligned}
\left|
\int_{\partial\Omega}(f_k-f)Tv_\varepsilon\,d\mathscr H
\right|
&\le
C\|f_k-f\|_{L^{\widetilde{\mathcal H}}(\partial\Omega)}
\|Tv_\varepsilon\|_{L^{\mathcal H}(\partial\Omega)}
\rightarrow 0.
\end{aligned}
\]
Therefore, \(\limsup_{k\to\infty}I(f_k) \le \mathcal J_f(v_\varepsilon) < I(f)+\varepsilon\). Letting \(\varepsilon\to 0\), we obtain \(\limsup_{k\to\infty}I(f_k)\le I(f)\). On the other hand, since
\[
\mathcal J_f(u_k)
=
\mathcal J_{f_k}(u_k)
+
\int_{\partial\Omega}(f_k-f)Tu_k\,d\mathscr H,
\]
the boundedness of the trace operator and the boundedness of \((u_k)\)
in \(N^{1,\mathcal H}(\Omega)\) imply
\[
\begin{aligned}
\left|
\int_{\partial\Omega}(f_k-f)Tu_k\,d\mathscr H
\right|
&\le
C\|f_k-f\|_{L^{\widetilde{\mathcal H}}(\partial\Omega)}
\|Tu_k\|_{L^{\mathcal H}(\partial\Omega)}
\rightarrow 0.
\end{aligned}
\]
Hence \(\mathcal J_f(u_k)-\mathcal J_{f_k}(u_k)\rightarrow0\), and, since \(\mathcal J_{f_k}(u_k)=I(f_k)\), \(\liminf_{k\to\infty}\mathcal J_f(u_k) = \liminf_{k\to\infty}I(f_k)\). By Remark \ref{Remark:WLSC}, \(\mathcal J_f(u)
\le
\liminf_{k\to\infty}\mathcal J_f(u_k)
=
\liminf_{k\to\infty}I(f_k)\). Since \(u\in N^{1,\mathcal H}_*(\Omega)\), \(I(f)\le\mathcal J_f(u)\). Consequently,
\( I(f)
\le
\mathcal J_f(u)
\le
\liminf_{k\to\infty}I(f_k) .\) Combining this with \(\limsup_{k\to\infty}I(f_k)\le I(f)\), we obtain \(I(f_k)\rightarrow I(f)\) and \(\mathcal J_f(u)=I(f)\). Finally, since \(u_k\rightharpoonup u\) weakly in the Banach space
\(N^{1,\mathcal H}(\Omega)\), Mazur's lemma yields, for every
\(k\in\mathbb N\), an integer \(N(k)\ge k\) and coefficients
\(\lambda_{j,k}\ge0\), \(j=k,\ldots,N(k)\), such that
\[
\sum_{j=k}^{N(k)}\lambda_{j,k}=1~~\text{and}~~ \sum_{j=k}^{N(k)}\lambda_{j,k}u_j
\rightarrow u~~\text{ in \(N^{1,\mathcal H}(\Omega)\)}.\]
Since \(N^{1,\mathcal H}_*(\Omega)\) is a linear subspace and
\(u_j\in N^{1,\mathcal H}_*(\Omega)\), each convex combination belongs
to \(N^{1,\mathcal H}_*(\Omega)\). Its strong limit therefore also
belongs to \(N^{1,\mathcal H}_*(\Omega)\). Thus \(\mathcal J_f(u)=I(f)\).
\end{proof}
\section*{Data Availability:}
No data was used for the research described in the article.
\small
\bibliographystyle{elsarticle-harv}
\bibliography{refrences}

\end{document}